\documentclass[12pt]{article}
\usepackage{amsmath}
\usepackage{amssymb}
\usepackage{latexsym}
\usepackage{amsthm}
\usepackage{tikz}
\usepackage{tabularx}
\usepackage{booktabs}
\usepackage{mathrsfs}
\usepackage{enumitem}
\usetikzlibrary{calc}
\usepackage{xcolor} % 添加颜色支持
\usepackage{ytableau}
\usepackage{graphicx}
\usepackage{algorithm,algorithmic}
\usepackage[colorlinks,
linkcolor=red,
anchorcolor=blue,
citecolor=green]{hyperref}

\newtheorem{thm}{Theorem}[section]
\newtheorem{lem}[thm]{Lemma}
\newtheorem{prop}[thm]{Proposition}
\newtheorem{cor}[thm]{Corollary}

\newtheorem{rem}[thm]{Remark}

\begin{document}
	\begin{center}
		{\large \bf Hurwitz Stability of Generalized Tur\'an Expressions of Polynomial Sequences}
	\end{center}

	\begin{center}
		Huihua Gao$^1$, Xin-Bei Liu$^2$\\[6pt]
        
		$^{1,2}$Center for Combinatorics, LPMC\\
		Nankai University, Tianjin 300071, P. R. China\\[6pt]
		
		Email: $^{1}${\tt huihuagao@mail.nankai.edu.cn}, 
			   $^{2}${\tt lxb@mail.nankai.edu.cn}
		\end{center}
	\noindent\textbf{Abstract.}
In this paper we prove the weak Hurwitz stability conjecture for the strong Tur\'an expressions of Eulerian polynomials, as well as the corresponding conjecture for Narayana polynomials, which were conjectured by Liu and Yan. We generalize the notion of Tur\'an expressions of polynomial sequences, and provide a criterion for the generalized stability of generalized Tur\'an expressions. Based on our sufficient condition, we not only confirm Liu and Yan's conjectures but also obtain the generalized stability of generalized Tur\'an expressions of a class of recurrent polynomial sequences and generalized Narayana polynomials.   


\noindent \emph{AMS Classification 2020:} 26C10, 05A15, 05A20

\section{Introduction}
A polynomial with real coefficients is called Hurwitz stable (resp. weakly Hurwitz stable) if all its zeros have negative (resp. nonpositive) real parts. 
In this paper, our main motivation is to prove Liu and Yan's recent stability conjectures concerning the Eulerian polynomials $\widetilde{A}_n(z)$ and the Narayana polynomials $N_n(z)$. Recall that $\widetilde{A}_n(z)$ and $N_n(z)$ can be defined respectively by 
\begin{align}\label{eq:Eulerian_recurrence}
   \widetilde{A}_n(z)
=nz\widetilde{A}_{n-1}(z)
+z(1-z)\widetilde{A}_{n-1}'(z)
\end{align}
with $\widetilde{A}_0(z)=1$ 
and 
\begin{align}\label{eq:Narayana_recurrence}
(n+2)N_{n+1}(z)
=(2n+1)(1+z)N_n(z)-(n-1)(1-z)^2N_{n-1}(z)
\end{align}
with $N_0(z)=1$ and $N_1(z)=z$.
Liu and Yan {\cite[Conjecture~4.1 and Conjecture~4.3]{Liu-Yan-2023}} conjectured that 
for any positive integers $1\le m\le n$, both
\begin{align*}
\widetilde{A}_m(z)\widetilde{A}_n(z)-\widetilde{A}_{m-1}(z)\widetilde{A}_{n+1}(z),    
\end{align*}
and 
\begin{align*}
N_m(z)N_n(z)-N_{m-1}(z)N_{n+1}(z)
\end{align*}
are weakly Hurwitz stable. 
Here we completely resolve Liu and Yan's conjectures by establishing a sufficient condition for stability of some generalization of Tur\'an expressions of polynomial sequences.  

Let us first give an overview of related backgrounds. 
The study of stability problems can be traced back to the classical
works of Routh~\cite{Routh-1877} and Hurwitz~\cite{Hurwitz-1895},
whose results led to the Routh--Hurwitz stability criterion. This
criterion is a fundamental tool for determining the asymptotic
stability of linear time-invariant systems~\cite{Barnett-Siljak-1977}.
The classical theory for real univariate polynomials was subsequently
extended to polynomials with complex
coefficients~\cite{Frank-1946}, while related notions of half-plane
stability were developed in the multivariate
setting~\cite{Borcea-Branden-2009,Choe-Oxley-Sokal-Wagner-2004}.
This broader theory has rich connections with many other areas of
mathematics such as matrix theory~\cite{Asner-1970,Gantmacher-1959,Kemperman-1982}, complex
analysis~\cite{Marden-1966}, and probability
theory~\cite{Borcea-Branden-Liggett-2009}.

Recently, Hurwitz stability has found several applications in combinatorics, including connections with semi-$\gamma$-positivity (see Ding and Zhu \cite{Ding-Zhu-2024}) and real-rootedness (see, for instance, Yang and Zhang \cite{Yang-Zhang-2015}). 
In this paper we focus on the study of Hurwitz stability of Tur\'an expressions of univariate polynomial sequences, which was initiated by Fisk~\cite[Section~21.11]{Fisk-2008}. Recall that the Tur\'an expressions $\mathcal{T}_n(f;z)$ of a polynomial sequence $(f_n(z))_{n\geq0}$ are defined by 
\begin{align*}
\mathcal{T}_n(f;z):=f_n(z)^2-f_{n-1}(z)f_{n+1}(z),\quad n\ge1,
\end{align*}
and their negative are known as the reverse Tur\'an expressions of $(f_n(z))_{n\geq0}$.
We say that $(f_n(z))_{n\geq0}$ is $z$-log-concave (resp. $z$-log-convex) if all coefficients of $\mathcal{T}_n(f;z)$ are nonnegative (resp. nonpositive) for every $n\geq1$. Liu and Wang \cite{Liu-Wang-2007} proved that both Bell polynomials and Eulerian polynomials are $z$-log-convex; see \cite[Propositions~4.4 and~4.6]{Liu-Wang-2007}. On the other hand, Fisk stated the weak Hurwitz stability of the associated Tur\'an expressions of these polynomial sequences in his unpublished book \cite{Fisk-2008}, a stronger property than $z$-log-convexity. However, there are some flaws in his proofs. 
For Bell polynomials the incorrectness of Fisk's approach was shown and then fixed by Chasse, Grabarek and Visontai \cite{Chasse-Grabarek-Visontai-2015}.
For Eulerian polynomials the incorrectness of Fisk's proof was fixed by Zhu \cite{Zhu-2018}, as explicitly pointed out by Ding and Zhu \cite{Ding-Zhu-2024}. For more results on the Hurwitz stability of Tur\'an expressions of combinatorial polynomial sequences, including the Lah polynomials, the alternating run polynomials of types $A$ and $B$, the up-down run polynomials, see \cite{Ding-Zhu-2024,Zhu-2018,Zhu-2021}.

Motivated by Fisk's work, it is natural to study the Hurwitz stability of strong Tur\'an expressions of strongly $z$-log-concave or $z$-log-convex sequences. By the strong Tur\'an expressions of a polynomial sequence $(f_n(z))_{n\geq0}$ we mean the following differences
\begin{align*}
\mathcal{T}_{n,m}(f;z)
:=f_m(z)f_n(z)-f_{m-1}(z)f_{n+1}(z),
\end{align*}
for all $1\leq m\leq n$. It is evident that $\mathcal{T}_{n,n}(f;z)=\mathcal{T}_{n}(f;z)$. 
We also call the negative of $\mathcal{T}_{n,m}(f;z)$ the reverse strong Tur\'an expressions of $(f_n(z))_{n\geq0}$.
Recall that $(f_n(z))_{n\geq0}$ is said to be strongly $z$-log-concave (resp. strongly $z$-log-convex) if all coefficients of $\mathcal{T}_{n,m}(f;z)$ are nonnegative (resp. nonpositive) for any $1\leq m\leq n$. 
Many combinatorial polynomial sequences are known to be strongly $z$-log-convex or strongly $z$-log-concave; see, for instance, \cite{Chen-Wang-Yang-2010,Chen-Wang-Yang-2011,Liu-Zhu-2015,Sagan-1992,Zhu-2013}.
It was Liu and Yan \cite{Liu-Yan-2023} who first investigated the Hurwitz stability of Tur\'an expressions of strongly $z$-log-convex sequences.  
In particular, they provided a unified approach to the Hurwitz stability of Tur\'an expressions of several  polynomial sequences, including the associated Lah polynomials, the Bell polynomials, the ordered Bell polynomials, the Dowling polynomials and the corresponding ordered Dowling polynomials.
However, their approach does not apply to the Eulerian polynomials and the Narayana polynomials. 
The main results of this paper are as follows, which were first conjectured by Liu and Yan \cite{Liu-Yan-2023}. 
Note that Liu and Yan's conjecture for the sequence of Narayana polynomials is stronger than its strong $q$-log-convexity proved by Chen, Wang, and Yang \cite[Theorem~3.2]{Chen-Wang-Yang-2010}.

\begin{thm}\label{Eulerian polynomial}
Let $\widetilde{A}_n(z)$ be the Eulerian polynomials defined by \eqref{eq:Eulerian_recurrence}. Then 
$\mathcal{T}_{n,m}(\widetilde{A};z)$ is weakly Hurwitz stable for any $1\leq m\leq n$.
\end{thm}

\begin{thm}\label{Narayana polynomial thm}
Let $N_n(z)$ be the Narayana polynomials defined by \eqref{eq:Narayana_recurrence}. Then 
$\mathcal{T}_{n,m}(N;z)$ is weakly Hurwitz stable for any $1\leq m\leq n$.
\end{thm}

In this paper we also generalize the above theorems in two directions. On the one hand, we will extend the notion of strong Tur\'an expressions to a broader concept. For this purpose, we recall some notation concerning integer partitions. 
Let $m$ be a positive integer. A partition of $m$ is a weakly decreasing sequence $\lambda=(\lambda_1,\lambda_2,\ldots)$ of nonnegative integers satisfying $|\lambda|=\sum_{i\ge1}\lambda_i=m$, denoted by $\lambda\vdash m$. Let $\ell(\lambda)$ denote the number of nonzero parts of $\lambda$. We may assume that 
$\lambda_i=0$ for any $i> \ell(\lambda)$. 
Given two partitions $\lambda,\mu\vdash m$, we say that $\lambda$ dominates $\mu$, written $\lambda\unrhd\mu$, if 
\begin{align*} 
\sum_{i=1}^j\lambda_i \ge \sum_{i=1}^j\mu_i \qquad\text{for all } j\ge1. 
\end{align*}
Given a sequence of real polynomials $(f_n(z))_{n\ge0}$ and a pair of partitions $(\lambda,\mu)$ of the same size with $\lambda\unrhd\mu$,  let
\begin{align*}
\mathcal{T}_{\lambda,\mu}(f;z) := \prod_{i=1}^s f_{\mu_i}(z)-\prod_{i=1}^s f_{\lambda_i}(z), 
\end{align*} 
where $s=\max\{\ell(\lambda),\ell(\mu)\}$. 
We call $\mathcal{T}_{\lambda,\mu}(f;z)$ the generalized Tur\'an expression of $(f_n(z))_{n\ge0}$ with respect to $\lambda$ and $\mu$. 
Clearly, for $1\leq m\leq n$ the strong Tur\'an expression $\mathcal{T}_{(n+1,m-1),(n,m)}(f;z)$ is the generalized Tur\'an expression $\mathcal{T}_{\lambda,\mu}(f;z)$ with $\lambda=(n+1,m-1)$ and $\mu=(n,m)$.

On the other hand, we will consider two wider classes of polynomial sequences.  
For Eulerian polynomials we follow Zhu~\cite{Zhu-2018} to consider the sequence $(E_n(z))_{n\ge 0}$ of polynomials with nonnegative coefficients satisfying the following recurrence relation: % $(E_n^{(\mathbf{a},\mathbf{b})}(z))_{n\ge 0}$:
\begin{align}\label{eq-fn-rec}
E_n(z)=[a_1n+a_2+(b_1n+b_2)z]E_{n-1}(z)+z(a_3+b_3z)E_{n-1}'(z),
\end{align}
with
\begin{align}\label{cond:parameter}
a_1, a_1+a_3, a_1+2a_3, b_1-|b_3| \ge 0; E_0(z)=1, \deg(E_n(z))=n. 
\end{align}
By~\eqref{eq:Eulerian_recurrence} the Eulerian polynomials $\widetilde{A}_n(z)$ form such a sequence  
with $a_1=a_2=b_2=0, a_3=b_1=1, b_3=-1$. 
For Narayana polynomials we consider the polynomial sequence $(N_n^{(r)}(z))_{n\geq0}$ (for some fixed $r>1/2$)
satisfying the following recurrence relation:
\begin{align}\label{eq:generalized_Narayana_re}
(n+2r)N_{n+1}^{(r)}(z)
=(2n+2r-1)(1+z)N_n^{(r)}(z)
-(n-1)(1-z)^2N_{n-1}^{(r)}(z)
\end{align}
with $N_0^{(r)}(z)=1$ and $N_1^{(r)}(z)=z$. 
Since $N_{n}(z)=N_{n}^{(1)}(z)$, we call $N_n^{(r)}(z)$ a generalized Narayana polynomial.
For more information on these polynomials and their connections to Gegenbauer polynomials, see Amdeberhan, Moll and Vignat ~\cite{Amdeberhan-2013}.

To state our general results on $(E_n(z))_{n\ge 0}$ and $(N_n^{(r)}(z))_{n\geq0}$, we also need the following definitions. Following Zhu~\cite{Zhu-2018}, call a real polynomial a generalized stable polynomial if each of its nonzero zeros has negative real part. Thus Hurwitz stability implies generalized stability, which in turn implies weak Hurwitz stability.
 Let $f(z)$ and $g(z)$ be real-rooted polynomials with zeros 
$\alpha_1\ge\cdots\ge\alpha_r,
\beta_1\ge\cdots\ge\beta_s,$
respectively. We write $f\preceq g$ if the zeros of $f$ interlace those of $g$, that is, either $s=r$ and
$\beta_1\ge\alpha_1\ge\beta_2\ge\alpha_2
\ge\cdots\ge\beta_r\ge\alpha_r,$
or $s=r+1$ and
$\beta_1\ge\alpha_1\ge\beta_2\ge\alpha_2
\ge\cdots\ge\beta_r\ge\alpha_r\ge\beta_s.$ Similarly, if all the inequalities are strict, we say that $f$ strictly interlaces $g$, and write $f\prec g$. By convention, nonzero constant polynomials strictly interlace all polynomials of degree at most one. We have the following results.

\begin{thm}\label{prop_first order}
Let $(E_n(z))_{n\ge 0}$ be a sequence of polynomials with nonnegative coefficients satisfying \eqref{eq-fn-rec} and \eqref{cond:parameter}.
If $E_n(z) \preceq  E_{n+1}(z)$ for any $n\geq 0$, then $\mathcal{T}_{\lambda,\mu}(E;z)$ is a generalized stable polynomial for any pair $(\lambda,\mu)$ of partitions with $\lambda\unrhd\mu$.
\end{thm}

\begin{thm}\label{prop_second order}
Let $(N_n^{(r)}(z))_{n\geq0}$ be the polynomial sequence defined by \eqref{eq:generalized_Narayana_re} for some $r>1/2$.
If $\lambda,\mu$ are two partitions satisfying $\lambda\unrhd\mu$, then $\mathcal{T}_{\lambda,\mu}(N^{(r)};z)$ is a generalized stable polynomial.
\end{thm}

The remaining part of this paper is organized as follows. 
In Section~\ref{sect-lemma} we will provide a general strategy to prove the stability of generalized Tur\'an expressions of stable polynomial sequences. 
Based on this strategy we will give a proof of Theorem \ref{prop_first order} and hence that of Theorem \ref{Eulerian polynomial}
in Section~\ref{Euler-type polynomials}. Finally, we will prove Theorem \ref{prop_second order} and hence Theorem \ref{Narayana polynomial thm} in Section \ref{Narayana polynomial}.

\section{Preliminaries}\label{sect-lemma}
The aim of this section is to provide a sufficient condition for  stability of generalized Tur\'an expressions. This criterion will serve as a reference point to prove Theorems \ref{prop_first order} and \ref{prop_second order}.

We first recall some basic facts about integer partitions. By the conjugate of a partition $\lambda=(\lambda_1,\lambda_2,\ldots)$, denoted by $\lambda'$, we mean the sequence
$(\lambda_1',\lambda_2',\ldots)$ defined by 
\begin{align}\label{eq-conjugate-parts}
\lambda_j'=\#\{i\geq1:\lambda_i\geq j\},\qquad  j\geq1.   
\end{align}
It is known that $\lambda\unrhd\mu$ if and only if $\mu'\unrhd\lambda'$; see \cite[Chapter~I, Eq.~(1.11)]{Macdonald-1995}. We need the following lemma.

\begin{lem} \label{partition}
Given two partitions $\lambda$ and $\mu$ of the same size, for each $j\geq 1$ let 
\begin{align*}
d_{j-1}(\lambda,\mu)=\sum_{k\geq j}(\lambda_k'-\mu_k').
\end{align*}
If $\lambda\unrhd\mu$,  then $d_0(\lambda,\mu)=d_M(\lambda,\mu)=0$ and $d_j(\lambda,\mu)\geq 0$ for any $1\leq j\leq M-1$, where 
$M=\max\{\lambda_1,\mu_1\}$. Moreover, if $\lambda$ and $\mu$ are two distinct partitions, then $d_j(\lambda,\mu)>0$ for some $1\leq j\leq M-1$.
\end{lem}

\begin{proof}
It is obvious that $d_0(\lambda,\mu)=d_M(\lambda,\mu)=0$
since $\lambda,\mu$ have the same size, and the definition of $M$. 
For each $1\leq j\leq M$, it is routine to verify that
\begin{align*}
d_j(\lambda,\mu)
=\sum_{k\geq j+1}(\lambda_k'-\mu_k')=-\sum_{k=1}^{j}(\lambda_k'-\mu_k')
=\sum_{k=1}^{j}(\mu_k'-\lambda_k')
\geq 0,
\end{align*}
where the last inequality follows from the fact $\mu'\unrhd\lambda'$.
By definition, if $\lambda\unrhd \mu$ but $\lambda\neq \mu$
then there exists some $1\leq j\leq M$ such that $d_j(\lambda,\mu)>0$.
\end{proof}
We proceed to state the sufficient condition for 
the stability of $\mathcal{T}_{\lambda,\mu}(f;z)$ of a real polynomial sequence $(f_n(z))_{n\geq0}$.
Given $\Omega\subseteq\mathbb{C}$, recall that a nonzero polynomial
$f\in\mathbb{R}[z]$ is said to be $\Omega$-stable if
$f(z)\neq0$ for every $z\in\Omega$. 
Throughout this paper, we assume that $\Omega$ contains infinitely many points in the complex plane $\mathbb{C}$. 
For more information on $\Omega$-stability, see Borcea and Br\"and\'en \cite{Borcea-Branden-2009}.
Let $\operatorname{Re}(z)$ and $\operatorname{Im}(z)$ denote the real and imaginary parts of a complex number $z$, respectively. 
Let $\mathcal{H}=\{z\in\mathbb{C}:\operatorname{Re}z\ge 0,z\neq 0\}$ throughout this paper.
Note that, when $\Omega=\mathcal{H}$, $\Omega$-stability becomes generalized stability due to Zhu \cite{Zhu-2018}. 
By convention, the zero polynomial is regarded as  $\Omega$-stable for any $\Omega\subseteq\mathbb{C}$. 
For a given sequence $(f_n(z))_{n\geq0}$ of real polynomials and each $m\geq 1$, let
\begin{align}\label{eq-polynomial-Q}
\mathcal{Q}_m(f;z)=\frac{f_{m+1}(z)f_{m-1}(z)}{f_m(z)^2}.
\end{align}
We find that the stability of $\mathcal{T}_{\lambda,\mu}(f;z)$
is closely related to the properties of the associated rational functions $\mathcal{Q}_m(f;z)$, as indicated below. 

\begin{prop}\label{prop_main result}
 Let $\lambda$ and $\mu$ be partitions. Given $\Omega\subseteq\mathbb{C}$ and a sequence 
$(f_n(z))_{n\geq0}$ of nonzero $\Omega$-stable polynomials,
let $\mathcal{Q}_n(f;z)$ be defined as in \eqref{eq-polynomial-Q}. 
If for each $n\geq1$ either $\mathcal{Q}_n(f;z)\equiv 1$ or
$|\mathcal{Q}_n(f;z)|>1$  on $\Omega$, then $\mathcal{T}_{\lambda,\mu}(f;z)$ is $\Omega$-stable
provided that $\lambda\unrhd\mu$.
\end{prop}

\begin{proof}
The case $\lambda=\mu$ is immediate, since
$\mathcal{T}_{\lambda,\mu}(f;z)\equiv 0$. Thus, we may assume that $\lambda$ and $\mu$ are distinct. 
Let $s=\max\{\ell(\lambda),\ell(\mu)\}$.
By setting $R_j(z)=f_j(z)/f_{j-1}(z)$ for each $j\geq1$, we obtain 
\begin{align*}
f_k(z)=f_0(z)\prod_{j=1}^{k}R_j(z),\qquad k\geq0.
\end{align*}
Consequently,
\begin{align*}
\prod_{i=1}^s f_{\lambda_i}(z)&=\prod_{i=1}^s \left(f_0(z)\prod_{j=1}^{\lambda_i}R_j(z)\right)=f_0(z)^s \prod_{i=1}^s\prod_{j=1}^{\lambda_i}R_j(z)\\[5pt]
&=f_0(z)^s\prod_{j=1}^{\lambda_1}R_j(z)^{\#\{i:\lambda_i\geq j\}}
=f_0(z)^s\prod_{j\geq 1}R_j(z)^{\lambda_j'},
\end{align*}
where the third equality is obtained by interchanging the order of multiplication and the last equality follows from 
\eqref{eq-conjugate-parts}. 
In the same manner, we obtain
\begin{align*}
\prod_{i=1}^s f_{\mu_i}(z)
=f_0(z)^s\prod_{j\geq 1}R_j(z)^{\mu_j'}.
\end{align*}
Let $d_i=d_i(\lambda,\mu)$ be as defined in Lemma \ref{partition}, which shows that $d_0=d_M=0$ and $d_j>0$ for some $1\leq j\leq M-1$ since $\lambda\unrhd \mu$. Thus,
\begin{align*}
\frac{\prod_{i=1}^s f_{\lambda_i}(z)}
     {\prod_{i=1}^s f_{\mu_i}(z)}
 =\prod_{j=1}^{M}R_j(z)^{d_{j-1}-d_j}
 =\prod_{j=1}^{M-1}
 \left(\frac{R_{j+1}(z)}{R_j(z)}\right)^{d_j}
 =\prod_{j=1}^{M-1}\mathcal{Q}_j(f;z)^{d_j}.
\end{align*}
By the condition on $\mathcal{Q}_j(f;z)$, we see that 
\begin{align*}
\left|
\frac{\prod_{i=1}^s f_{\lambda_i}(z_0)}
{\prod_{i=1}^s f_{\mu_i}(z_0)}
\right|
=\prod_{j\ge1}|\mathcal{Q}_j(f;z_0)|^{d_j}
\ge 1
\end{align*}
for any $z_0\in \Omega$. 
If the above inequality is strict, then $\mathcal{T}_{\lambda,\mu}(f;z_0)\neq 0$. 
Otherwise, we have $|\mathcal{Q}_j(f;z_0)|=1$ for each $j\geq 1$ with $d_j>0$, which forces $\mathcal{Q}_j(f;z)\equiv 1$ on $\Omega$ by the hypothesis and hence $\mathcal{T}_{\lambda,\mu}(f;z)\equiv 0$.
This completes the proof.
\end{proof}

\begin{rem}
The conclusion of Proposition~\ref{prop_main result} remains valid if
we require that $\mathcal{Q}_n(f;z)$ satisfy the condition that 
for each $n\geq1$ either $\mathcal{Q}_n(f;z)\equiv1$ or $|\mathcal{Q}_n(f;z)|<1$ on $\Omega$.
\end{rem}
\section{Eulerian polynomial}\label{Euler-type polynomials}

In this section we aim to prove Theorem~\ref{prop_first order}. 
As a result, we extend Zhu's result \cite{Zhu-2018} from Tur\'an expressions to generalized Tur\'an expressions for a class of recurrent polynomial sequences, including that of colored Eulerian polynomials and that of Eulerian polynomials $\widetilde{A}_n(z)$ as special cases. In particular, we obtain Theorem \ref{Eulerian polynomial} and thus confirm Liu and Yan's Conjecture  \cite[Conjecture~4.1]{Liu-Yan-2023} on the weak Hurwitz stability of reverse strong Tur\'an expressions for Eulerian polynomials.

In order to use Proposition~\ref{prop_main result} to prove Theorem~\ref{prop_first order}, we need to take 
$\Omega=\mathcal{H}$. 
The remainder of the proof is analogous to that of \cite[Theorem 1.6]{Zhu-2018}.
For the sake of completeness, we include a result due to Fisk \cite{Fisk-2008} below, which allows us to give useful expressions of the quotients ${E_{n-1}(z)}/{E_{n}(z)}$ for the involved sequence $(E_n(z))_{n\ge 0}$ in 
Theorem~\ref{prop_first order}.

\begin{lem}[{\cite[Lemma 1.20]{Fisk-2008}}]\label{lem_strict_interlace_reduction}
Let $f(z)$ and $g(z)$ be two nonzero real-rooted polynomials with positive leading coefficients. Assume that $\deg(f(z))=\deg(g(z))+1$
and $g(z)\prec f(z)$. 
 If the distinct zeros of $f(z)$ are $\gamma_1,\ldots,\gamma_t$, then
\begin{align*}
{g}(z)=\sum_{i=1}^{t}c_i\frac{{f}(z)}{z-\gamma_i},
\end{align*}
where
$c_i>0$
for all $i$.
\end{lem}

We are now ready to prove Theorem~\ref{prop_first order}. 

\begin{proof}[Proof of Theorem~\ref{prop_first order}]
By hypothesis, for any $n \geq 0$ the polynomial $E_n(z)$ is $\mathcal{H}$-stable. Let $\mathcal{Q}_n(E;z)$ be the rational function with respect to
the sequence $(E_n(z))_{n\ge 0}$ as defined by \eqref{eq-polynomial-Q}. 
By Proposition~\ref{prop_main result} it suffices to show that for each $n\geq 1$ either 
$\mathcal{Q}_n(E;z)\equiv 1$ or $|\mathcal{Q}_n(E;z)|>1$ over $\mathcal{H}$.
To this end, let 
\begin{align*}
\widetilde{E}_n(z)=\frac{E_n(z)}{\gcd(E_n(z),E_{n-1}(z))},\qquad
\widehat{E}_n(z)=\frac{E_{n-1}(z)}{\gcd(E_n(z),E_{n-1}(z))}.
\end{align*}
Then from ${E}_{n-1}(z) \preceq {E}_{n}(z)$ it follows that 
$\widehat{E}_n(z)$ strictly interlaces $ \widetilde{E}_n(z)$.
Thus all zeros of $\widetilde{E}_n(z)$ are distinct and nonpositive, say $\gamma_1,\ldots,\gamma_t$.

By Lemma~\ref{lem_strict_interlace_reduction} we have
\begin{align}\label{eq-quotient}
\frac{E_{n-1}(z)}{E_n(z)}
=\frac{\widehat{E}_n(z)}{\widetilde{E}_n(z)}
=\sum_{i=1}^t\frac{c_i}{z-\gamma_i},
\end{align}
where $c_i>0$ for $1\le i\le t$.
Based on the recurrence relation \eqref{eq-fn-rec} and along the lines of the proof of Zhu~\cite[Theorem~1.6]{Zhu-2018}, one can check that
\begin{align*}
\mathcal{Q}_n(E;z)-1 
&= (a_1+ b_1z) \frac{E_{n-1}(z)}{E_n(z)} - z(a_3+ b_3z) \left( \frac{E_{n-1}(z)}{E_n(z)} \right)'.
\end{align*}
By further using \eqref{eq-quotient}, we find that 
\begin{align*}
\mathcal{Q}_n(E;z)-1 &= (a_1 + b_1z) \sum_{i=1}^{t} \frac{c_i}{z - \gamma_i} + z(a_3 + b_3 z) \sum_{i=1}^{t} \frac{c_i}{(z - \gamma_i)^2}.
\end{align*}
By setting $A=a_1+a_3$ and $B=b_1+b_3$, which are nonnegative by the hypothesis, we can rewrite the above equality as
\begin{align}\label{eq-QnEz-1}
\mathcal{Q}_n(E;z)-1 &=\sum_{i=1}^t
c_i\frac{Bz^2+(A-b_1\gamma_i)z-a_1\gamma_i}{(z-\gamma_i)^2}.
\end{align}
For every $\gamma_i\le 0$ ($\gamma_i\notin \mathcal{H}$) and $z=x+Iy\in\mathcal{H}$ (hence $x\geq 0$) it is routine to verify that 
\begin{align}
&\mathrm{Re}\left(\overline{(z-\gamma_i)}^2\left(
Bz^2+(A-b_1\gamma_i)z-a_1\gamma_i\right)\right)=\nonumber\\
&B(x^2-y^2)^2
+Bx^2(\gamma_i^2-2x\gamma_i)
+(xA-xb_1\gamma_i-a_1\gamma_i)(x-\gamma_i)^2\nonumber\\
&+y^2\left[
4Bx^2+xA-(2B+b_1)x\gamma_i
-(a_1+2a_3)\gamma_i+(b_1-b_3)\gamma_i^2
\right]\ge0,\label{eq-Realpart}
\end{align}
where we let $\overline{z}$ denote the conjugate of $z$ as usual. 
Since $\widetilde{E}_n(z)$ is $\mathcal{H}$-stable, we have $z-\gamma_i\neq 0$ and hence 
\begin{align*}
\mathrm{Re} (\mathcal{Q}_n(E;z)-1)= \sum_{i=1}^t \mathrm{Re} \, \left(c_i\frac{Bz^2+(A-b_1\gamma_i)z-a_1\gamma_i}{(z-\gamma_i)^2}\right)
\ge 0
\end{align*}
for any $z\in\mathcal{H}$.

We proceed to show that 
either $\mathcal{Q}_n(E;z)\equiv 1$ or $|\mathcal{Q}_n(E;z)|>1$ over $\mathcal{H}$.
There are two cases to consider according to whether $\widetilde{E}_n(z)$ has a negative zero or not. 
If it has no negative zeros, then $t=1$ and $\gamma_1=0$.
Hence
\begin{align*}
\mathcal{Q}_n(E;z)-1=c_1\left(B+\frac{A}{z}\right).
\end{align*}
If $A=B=0$, then $\mathcal{Q}_n(E;z)\equiv1$. Otherwise, it is clear that 
$|\mathcal{Q}_n(E;z)-1|>0$ and hence
\begin{align*}
|\mathcal{Q}_n(E;z)|^2
&=1+2\operatorname{Re}\bigl(\mathcal{Q}_n(E;z)-1\bigr)
  +|\mathcal{Q}_n(E;z)-1|^2>1,
\end{align*}
implying $|\mathcal{Q}_n(E;z)|>1$ over $\mathcal{H}$.

We next consider the case that $\widetilde{E}_n(z)$ has at least one negative zero, say $\gamma_i<0$ for some $i$.
We may assume that at
least one of $A$ and $b_1$ is positive. Otherwise, we have $a_1=a_3=0$
 since $0=2A=a_1+(a_1+2a_3)=0, a_1\geq 0$ and $a_1+2a_3\geq 0$, and we also have 
 $B=0$ since $|b_3|\leq b_1=0$ and $B=b_1+b_3$; this forces that $\mathcal{Q}_n(E;z)\equiv 1$ in view of 
\eqref{eq-QnEz-1}. If $A>0$ or $b_1>0$, then at least one of
$a_1$, $a_1+2a_3$, $B$, and $b_1-b_3$ is positive. % in view of $2A=a_1+(a_1+2a_3)$ and $2b_1=B+(b_1-b_3)$.
In this case we must have $|\mathcal{Q}_n(E;z)|>1$ over $\mathcal{H}$. To prove this, it suffices to show that 
\begin{align}\label{eq-RealpartInequality}
\operatorname{Re}\left(
\frac{Bz^2+(A-b_1\gamma_i)z-a_1\gamma_i}
{(z-\gamma_i)^2}
\right)>0
\end{align}
for each $z=x+yI\in \mathcal{H}$. 
Our proof of the above inequality will be divided into two cases: $x>0$ or $x=0$.
If $x>0$, then it follows from \eqref{eq-Realpart} that
\begin{align*}
\mathrm{Re}&\left(\overline{(z-\gamma_i)}^2\left(
Bz^2+(A-b_1\gamma_i)z-a_1\gamma_i\right)\right)\\[5pt]
& \geq 
(xA-xb_1\gamma_i-a_1\gamma_i)(x-\gamma_i)^2>0,
\end{align*}
and hence \eqref{eq-RealpartInequality} holds, where the second inequality follows from $A>0$ or $b_1>0$.
We proceed to consider the case of $x=0$, for which we must have $y\neq 0$ since $z=x+yI\in \mathcal{H}$.
Again by \eqref{eq-Realpart} we get that 
\begin{align*}
\mathrm{Re}&\left(\overline{(z-\gamma_i)}^2\left(
Bz^2+(A-b_1\gamma_i)z-a_1\gamma_i\right)\right)\\[5pt]
& = 
By^4-a_1\gamma_i^3
+y^2\left[-(a_1+2a_3)\gamma_i
+(b_1-b_3)\gamma_i^2\right]>0,
\end{align*}
and hence \eqref{eq-RealpartInequality} also holds, where the last inequality follows from the fact that
at least one of $a_1$, $a_1+2a_3$, $B$, and $b_1-b_3$ is positive.
This completes the proof. 
\end{proof}

As an immediate consequence of Theorem \ref{prop_first order}, we have the following result, which implies Theorem \ref{Eulerian polynomial}.

\begin{cor}
\label{cor_tilde_A_Eulerian}
For any pair of partitions $\lambda$ and $\mu$ 
with $\lambda\unrhd\mu$, the Tur\'an expression
$\mathcal{T}_{\lambda,\mu}(\widetilde{A};z)$ is a generalized stable polynomial.
\end{cor}

\begin{proof}
The polynomial $\widetilde{A}_n(z)$ has nonnegative coefficients and $\deg\widetilde{A}_n(z)=n$ for any $n\geq0$.  We know that $\widetilde{A}_{n-1}(z)\preceq \widetilde{A}_{n}(z)$ for any $n\geq1$ (see \cite[Proposition~1.7]{Zhu-2018} for instance). According to \eqref{eq:Eulerian_recurrence}, the desired result follows from Theorem~\ref{prop_first order} by setting $a_1=a_2=b_2=0$, $b_1=a_3=1$, and $b_3=-1$.
\end{proof}

More generally, Proposition~\ref{prop_main result} and Theorem~\ref{prop_first order} also allow us to obtain the generalized stability of Tur\'an expressions of the sequence $(A_n^{(r)}(z))_{n\geq0}$ of $r$-colored Eulerian polynomials, introduced by Steingr\'{\i}msson~\cite{Steingrimsson-1994} for any positive integer $r\geq 1$  and defined by 
\begin{align} \label{eq-r-colored-EulerPol}
A_{n+1}^{(r)}(z)
=\bigl((rn+r-1)z+1\bigr)A_n^{(r)}(z)
+rz(1-z)\left(A_n^{(r)}(z)\right)',
\end{align}
with $A_0^{(r)}(z)=1$. We would like to point out that $\widetilde{A}_n(z)=zA_n^{(1)}(z)$ for any $n\ge1$,
which follows immediately from the corresponding combinatorial
definitions; see Stanley~\cite[Section~1.4, Eq.~(1.36)]{Stanley-2012}
and Steingr\'{\i}msson~\cite[Definition~5]{Steingrimsson-1994}. Alternatively, this relation can be verified by induction using \eqref{eq:Eulerian_recurrence} and \eqref{eq-r-colored-EulerPol}. The combinatorial definition given by
Steingr\'{\i}msson~\cite[Definition~5]{Steingrimsson-1994} also reveals that for every $n\geq0$ all coefficients of $A_n^{(r)}(z)$ are nonnegative.  
Moreover, Steingr\'{\i}msson~\cite[Theorem~19]{Steingrimsson-1994} proved that $A_n^{(r)}(z)$   has only real zeros. Since its coefficients are
nonnegative and $A_n^{(r)}(0)=1$, all these zeros are negative. Hence it is $\mathcal{H}$-stable for any $n\geq 0$. For further results on the $r$-colored Eulerian polynomials, see \cite{Savage-Visontai-2015,Liu-Yan-2025}. We have the following result, which is stronger than the $z$-log-convexity of $(A_n^{(r)}(z))_{n\geq0}$ implicitly stated by Zhu~\cite[Theorem 1.6]{Zhu-2018} and even stronger than its strong $z$-log-convexity established by Liu and Zhu \cite[Propositions 3.1 and 3.4]{Liu-Zhu-2015}. 

\begin{cor}
\label{cor_colored_Eulerian}
Let $r$ be a positive integer. For any pair of partitions
$\lambda$ and $\mu$ with $\lambda\unrhd\mu$, the Tur\'an expression
$\mathcal{T}_{\lambda,\mu}(A^{(r)};z)$ is a generalized stable polynomial.
\end{cor}

\begin{proof}
Let us first consider the case of $r\geq2$. By \eqref{eq-r-colored-EulerPol} and $A_0^{(r)}(z)=1$, it is easy to check that $\deg A_n^{(r)}(z)=n$ for any $n\geq0$. Moreover, by \cite[Proposition~1.7]{Zhu-2018}, we have $A_{n-1}^{(r)}(z)\preceq A_{n}^{(r)}(z)$ for any $n\ge 1$. Taking 
\begin{align*}
a_1=0,\qquad a_2=1,\qquad a_3=r,\qquad
b_1=r,\qquad b_2=-1,\qquad b_3=-r
\end{align*}
in Theorem~\ref{prop_first order} leads to the desired result.

We proceed to consider the case of $r=1$. In this case, by \eqref{eq-r-colored-EulerPol} we can check that $\deg A_n^{(1)}(z)=n-1$ for $n\geq1 $. By using \cite[Proposition~1.7]{Zhu-2018} again, we see that $A_{n}^{(1)}(z)\preceq A_{n+1}^{(1)}(z)$ for any $n\ge 1$. The arguments in the proof of Theorem~\ref{prop_first order} carry over directly to the sequence 
$(A_n^{(1)}(z))_{n\geq 1}$, from which we derive that for any $n\geq 2$ either
$\mathcal{Q}_n(A^{(1)};z)\equiv1$ or $|\mathcal{Q}_n(A^{(1)};z)|>1$ over $\mathcal{H}$. 
Moreover, it is routine to check that
\begin{align*}
\mathcal{Q}_1(A^{(1)};z)
=\frac{ A_2^{(1)}(z) A_0^{(1)}(z)}{ A_1^{(1)}(z)^2}
=1+z,
\end{align*}
and hence 
\begin{align*}
|\mathcal{Q}_1(A^{(1)};z)|
=\sqrt{(1+z)(1+\overline{z})}=\sqrt{1+2\operatorname{Re}z+|z|^2}>1
\end{align*}
for every $z\in\mathcal{H}$.
Together with the fact that for any $n\geq 0$ the polynomial $A_n^{(1)}(z)$ is $\mathcal{H}$-stable, 
we can derive the generalized stability of $\mathcal{T}_{\lambda,\mu}(A^{(1)};z)$ from 
Proposition~\ref{prop_main result}, as desired. This completes the proof. 
\end{proof}

\section{Narayana polynomial}\label{Narayana polynomial}

The goal of this section is to prove Theorem~\ref{Narayana polynomial thm} and Theorem~\ref{prop_second order}.
Note that the former is a special case of the latter. 

Following the same strategy used in the proof of Theorem \ref{prop_first order}, we again apply Proposition~\ref{prop_main result} to prove Theorem~\ref{prop_second order}.  
To this end, we need the $\mathcal{H}$-stability of the
generalized Narayana polynomial $N_n^{(r)}(z)$ $(r>1/2)$, which can be derived from
\cite[Eq.~(6.10)]{Amdeberhan-2013} and the fact that all zeros
of Gegenbauer polynomials lie in $(-1,1)$(see, for example,
Szegő~\cite{Szego-1975}). We further need to show that 
$\mathcal{Q}_n({N^{(r)};z})\equiv 1$ or $|\mathcal{Q}_n({N^{(r)};z})|>1$ over $\mathcal{H}$.
To prove this, we introduce a sequence $(P_n^{(r)}(z))_{n\geq 0}$ ($r>1/2$) of polynomials, which satisfy the following recurrence relation
\begin{equation}\label{eq:F_recurrence}
P_{n+1}^{(r)}(z)
=\theta_n^{(r)}P_n^{(r)}(z)-\eta_n^{(r)}zP_{n-1}^{(r)}(z),\qquad n\geq 1
\end{equation}
where
\begin{align*}
  P_0^{(r)}(z)=P_1^{(r)}(z)=1,\qquad \theta_n^{(r)}=\frac{2n+2r-1}{n+2r},
\qquad
\eta_n^{(r)}=\frac{n-1}{n+2r}.  
\end{align*}
By \eqref{eq:F_recurrence} one can see that 
\begin{align}\label{eq-pvalueat0}
P_0^{(r)}(0)=P_0^{(r)}(1)=1,\quad P_n^{(r)}(0)=\prod_{i=1}^{n-1} \theta_i^{(r)}>0,\quad P_n^{(r)}(1)=1
\end{align}
for any $n\geq 1$ and $r>1/2$.
By comparing \eqref{eq:F_recurrence} and \eqref{eq:generalized_Narayana_re} and setting 
$u(z)=(1-z)^2/(1+z)^2$, we find that for any $n\geq 1$
\begin{equation}\label{eq:N_F_relation}
N_n^{(r)}(z)
=z(1+z)^{n-1}P_n^{(r)}(u(z)),
\qquad z\neq-1.
\end{equation}
As will be shown in subsequent proofs, it is more convenient to work with $P_n^{(r)}(z)$ instead of $N_n^{(r)}(z)$.
To simplify the recurrence~\eqref{eq:F_recurrence}, we introduce the 
normalized polynomials
$\widehat P_n^{(r)}(z)
={P_n^{(r)}(z)}/{P_n^{(r)}(0)}$.
Then the  recurrence~\eqref{eq:F_recurrence} can be rewritten as
\begin{align}\label{eq:normalized_F_recurrence}
\widehat P_{n+1}^{(r)}(z)=\widehat P_n^{(r)}(z)-\rho_n^{(r)}z\widehat P_{n-1}^{(r)}(z), \qquad n\geq 1
\end{align}
where
\begin{align}\label{eq_rho}
\rho_n^{(r)}
=\frac{\eta_n^{(r)}P_{n-1}^{(r)}(0)}
{\theta_n^{(r)}P_n^{(r)}(0)}
=\frac{(n-1)(n+2r-1)}
{(2n+2r-1)(2n+2r-3)}>0,
\qquad n\geq2.
\end{align}
In particular, $\rho_1^{(r)}=0$.
And it can be checked that
\begin{align}\label{eq:rho_difference}
\rho_n^{(r)}-\rho_{n-1}^{(r)}=\frac{(2r-1)(2r+1)}{(2n+2r-5)(2n+2r-3)(2n+2r-1)}>0, \quad n\geq2.
\end{align}
The polynomials $P_n^{(r)}(z)$ and their normalized forms
$\widehat P_n^{(r)}(z)$ satisfy the following properties, which will
play a key role in establishing the required modulus condition for
$\mathcal Q_n(N^{(r)};z)$.
\begin{lem}\label{lem:F_properties}
For any $n\geq3$ and $|z|\leq1$ with $z\neq1$, we have
$|{P_{n}^{(r)}(z)}|>|{P_{n-1}^{(r)}(z)}|$. Moreover, for any $2\leq j\leq n$, we have
\begin{align*}
\left|
\frac{\widehat P_{j-1}^{(r)}(z)}
{\widehat P_n^{(r)}(z)}
\right|
<\frac{\widehat P_{j-1}^{(r)}(1)}
{\widehat P_n^{(r)}(1)}.
\end{align*}
\end{lem}

\begin{proof}
Let us first use induction on $n$ to prove the first statement. Based on \eqref{eq:F_recurrence}, one can obtain that $P_2^{(r)}(z)=1$ and
$P_3^{(r)}(z)=1+(1-z)/(2r+2)$. Clearly, we have 
$|P_3^{(r)}(z)|\geq\operatorname{Re}P_3^{(r)}(z)>1$ whenever $|z|\leq1$ with $z\neq1$, as asserted.
Suppose that the assertion holds for some $n\geq3$. By \eqref{eq-pvalueat0} it is clear that 
$|P_{n+1}^{(r)}(0)|>|P_n^{(r)}(0)|$. 
Now we may assume that $z\neq 0$. Using~\eqref{eq:F_recurrence} and $\eta_n^{(r)}>0, \theta_n^{(r)}>0$, we get
\begin{align*}
|P_{n+1}^{(r)}(z)|&=|\theta_n^{(r)}P_n^{(r)}(z)
-\eta_n^{(r)}z P_{n-1}^{(r)}(z)|\\
&\geq \theta_n^{(r)}|P_n^{(r)}(z)|
-\eta_n^{(r)}|z|\,|P_{n-1}^{(r)}(z)|\\
&>(\theta_n^{(r)}-\eta_n^{(r)})|P_n^{(r)}(z)|\\
&=|P_n^{(r)}(z)|,
\end{align*}
where the strict inequality follows from the induction hypothesis. 

We continue to prove the second statement. Note that for any $k\geq 3$ we have
\begin{align}\label{eq:normalized_ratio_bound}
\left|\frac{\widehat P_{k-1}^{(r)}(z)}{\widehat P_k^{(r)}(z)}\right|
=\frac{P_k^{(r)}(0)}{P_{k-1}^{(r)}(0)}\left|\frac{P_{k-1}^{(r)}(z)}{P_k^{(r)}(z)}\right|
<\frac{P_k^{(r)}(0)}{P_{k-1}^{(r)}(0)}
=\frac{\widehat P_{k-1}^{(r)}(1)}{\widehat P_k^{(r)}(1)},
\end{align}
where the inequality follows from the first statement we just proved. 
Since $\widehat P_1^{(r)}(z)=\widehat P_2^{(r)}(z)=1$,  for $2\leq j\leq n$ and $n\geq3,$ we have

\begin{align*}
\left|
\frac{\widehat P_{j-1}^{(r)}(z)}
{\widehat P_n^{(r)}(z)}
\right|
&=\prod_{k=j}^{n}
\left|
\frac{\widehat P_{k-1}^{(r)}(z)}
{\widehat P_k^{(r)}(z)}
\right|<\prod_{k=j}^{n}
\frac{\widehat P_{k-1}^{(r)}(1)}
{\widehat P_k^{(r)}(1)}
=
\frac{\widehat P_{j-1}^{(r)}(1)}
{\widehat P_n^{(r)}(1)},
\end{align*}
where the inequality is strict by
\eqref{eq:normalized_ratio_bound} with $k=n\geq3$.
 This completes the proof.
\end{proof}

We proceed to use $P_n^{(r)}(z)$ (or the normalized form $\widehat P_{n}^{(r)}(z)$) to study the modulus condition for 
$\mathcal{Q}_n(N^{(r)};z)$. 
Let $\mathcal{D}=\{z| |z|\leq 1 \mbox{ and }z\neq 1\}\subseteq \mathbb{C}$. 
By using \eqref{eq:N_F_relation} and taking $u=u(z)={(1-z)^2}/{(1+z)^2}$, for every 
$n\geq2$ one can show that 
\begin{align*} 
\mathcal{Q}_n(N^{(r)};z) 
=\frac{N_{n+1}^{(r)}(z)N_{n-1}^{(r)}(z)} 
        {\bigl(N_n^{(r)}(z)\bigr)^2} 
=\frac{P_{n+1}^{(r)}(u)P_{n-1}^{(r)}(u)} 
        {\bigl(P_n^{(r)}(u)\bigr)^2} 
 =\mathcal{Q}_n(P^{(r)};u).
\end{align*} 
Thus, for $n\geq 2$,  the modulus condition for 
$\mathcal{Q}_n(N^{(r)};z)$ over $\mathcal{H}$ is equivalent to the modulus condition for 
$\mathcal{Q}_n(P^{(r)};z)$ but over $\mathcal{D}$. 
To study the latter, we recall the Tur\'an expressions $\mathcal{T}_n(\widehat P^{(r)};z)=(\widehat{P}_{n}^{(r)}(z))^2 -\widehat{P}_{n+1}^{(r)}(z)\widehat{P}_{n-1}^{(r)}(z)$ and use \eqref{eq-pvalueat0} to rewrite $\mathcal{Q}_n({P^{(r)};z})$ as
\begin{align}\label{Q_n(P)}
\mathcal{Q}_n({P^{(r)};z})=\frac{\theta_n^{(r)}}{\theta_{n-1}^{(r)}}\left[1-\frac{\mathcal{T}_n(\widehat P^{(r)};z)}{\left(\widehat P_n^{(r)}(z)\right)^2}\right].
\end{align}
We need the following technical lemma for
$\mathcal T_n(\widehat P^{(r)};z)$.

\begin{lem}\label{lemma-Turan}
For every $n\geq2$, we have 
\begin{align*}
\mathcal T_n(\widehat P^{(r)};z)
={}&\sum_{j=2}^{n}
\left(\rho_j^{(r)}-\rho_{j-1}^{(r)}\right)
\left(\prod_{k=j}^{n-1}\rho_k^{(r)}\right)
z^{n-j+1}
\left(\widehat P_{j-1}^{(r)}(z)\right)^2.
\end{align*}
\end{lem}

\begin{proof}
Applying \eqref{eq:normalized_F_recurrence} twice 
leads to 
\begin{align}
\mathcal T_n(\widehat P^{(r)};z)
={}&\left(\widehat P_n^{(r)}(z)\right)^2
-\widehat P_{n-1}^{(r)}(z)
 \left(
 \widehat P_n^{(r)}(z)
 -\rho_n^{(r)}z\widehat P_{n-1}^{(r)}(z)
 \right)
\notag\\
={}&\widehat P_n^{(r)}(z)
\left(
\widehat P_n^{(r)}(z)
-\widehat P_{n-1}^{(r)}(z)
\right)
+\rho_n^{(r)}z
\left(\widehat P_{n-1}^{(r)}(z)\right)^2
\notag\\
={}&\rho_n^{(r)}z
\left(\widehat P_{n-1}^{(r)}(z)\right)^2-\rho_{n-1}^{(r)}z
\widehat P_n^{(r)}(z)\widehat P_{n-2}^{(r)}(z).\notag
\end{align}
Further splitting $\rho_n^{(r)}$ into $\rho_{n-1}^{(r)}$ and $\rho_n^{(r)}-\rho_{n-1}^{(r)}$ gives 
\begin{align}
\mathcal T_n(\widehat P^{(r)};z)={}&\rho_{n-1}^{(r)}z
\left[
\left(\widehat P_{n-1}^{(r)}(z)\right)^2
-\widehat P_n^{(r)}(z)\widehat P_{n-2}^{(r)}(z)
\right]
\notag\\
&+\left(\rho_n^{(r)}-\rho_{n-1}^{(r)}\right)z
\left(\widehat P_{n-1}^{(r)}(z)\right)^2
\notag\\
={}&\rho_{n-1}^{(r)}z
\mathcal T_{n-1}(\widehat P^{(r)};z)
+\left(\rho_n^{(r)}-\rho_{n-1}^{(r)}\right)z
\left(\widehat P_{n-1}^{(r)}(z)\right)^2.
\label{eq:T_recurrence}
\end{align}
Since
$\mathcal T_1(\widehat P^{(r)};z)
=(\widehat P_1^{(r)}(z))^2
-\widehat P_0^{(r)}(z)\widehat P_2^{(r)}(z)
=0,$
iterating \eqref{eq:T_recurrence} yields
\begin{align*}
\mathcal T_n(\widehat P^{(r)};z)
={}&
\left(\prod_{k=1}^{n-1}\rho_k^{(r)}\right)
z^{n-1}\mathcal T_1(\widehat P^{(r)};z)\\
&+\sum_{j=2}^{n}
\left(\rho_j^{(r)}-\rho_{j-1}^{(r)}\right)
\left(\prod_{k=j}^{n-1}\rho_k^{(r)}\right)
z^{n-j+1}
\left(\widehat P_{j-1}^{(r)}(z)\right)^2,
\end{align*}
as desired. This completes the proof. 
\end{proof}

The following lemma provides the  modulus condition for the rational functions $\mathcal Q_n(P^{(r)};z)$ over
$\mathcal D$.

\begin{lem}\label{lem:L_estimate}
Let $r>1/2$ and let $\mathcal{D}$ be defined as above. Then, for $n\geq2$ and $z\in \mathcal{D}$, we have
\begin{align*}
|\mathcal{Q}_n({P^{(r)};z})|\geq \operatorname{Re}\mathcal{Q}_n({P^{(r)};z})>1.
\end{align*}
\end{lem}

\begin{proof}  We first consider the case $n=2$. 
Since $P_1^{(r)}(z)=P_2^{(r)}(z)=1$, we have $\mathcal{Q}_2({P^{(r)};z})=\theta_2^{(r)}-\eta_2^{(r)}z=1+(1-z)/(2r+2)$. 
Thus, $\operatorname{Re}(\mathcal{Q}_2({P^{(r)};z})-1)=(1-\operatorname{Re}z)/(2r+2)>0$ whenever $z\in\mathcal{D}$.

We proceed to prove the assertion for $n\geq3$. It suffices to prove that $\operatorname{Re}\left(\mathcal{Q}_n({P^{(r)};z})-1\right)>0$ for any $z\in\mathcal{D}$. It follows from \eqref{Q_n(P)} and
$\mathcal Q_n(P^{(r)};1)=1$ that
\begin{align*}
\mathcal{Q}_n(P^{(r)};z)-1
&=\mathcal{Q}_n(P^{(r)};z)-
\mathcal{Q}_n(P^{(r)};1)\nonumber \\ 
&=\frac{\theta_n^{(r)}}{\theta_{n-1}^{(r)}}
\left[\frac{\mathcal{T}_n(\widehat P^{(r)};1)}{\left(\widehat P_n^{(r)}(1)\right)^2}-
\frac{\mathcal{T}_n(\widehat P^{(r)};z)}{\left(\widehat P_n^{(r)}(z)\right)^2}
\right].
\end{align*}
One can easily check that ${\theta_n^{(r)}}/{\theta_{n-1}^{(r)}}>0$ for any $n\geq 3$. We only need to prove that 
if $z\in\mathcal{D}$ then
\begin{align}\label{Re>0}
\operatorname{Re}\left[\frac{\mathcal{T}_n(\widehat P^{(r)};1)}{\left(\widehat P_n^{(r)}(1)\right)^2}-
\frac{\mathcal{T}_n(\widehat P^{(r)};z)}{\left(\widehat P_n^{(r)}(z)\right)^2}
\right]>0.
\end{align}
By using Lemma \ref{lemma-Turan}, we obtain 
\begin{align*}
\frac{\mathcal{T}_n(\widehat P^{(r)};1)}{\left(\widehat P_n^{(r)}(1)\right)^2}-\frac{\mathcal{T}_n(\widehat P^{(r)};z)}{\left(\widehat P_n^{(r)}(z)\right)^2}&=\sum_{j=2}^{n}\left(\rho_j^{(r)}-\rho_{j-1}^{(r)}\right)\left(\prod_{k=j}^{n-1}\rho_k^{(r)}\right)\notag\\
&\quad\times\left[\left(\frac{\widehat P_{j-1}^{(r)}(1)}{\widehat P_n^{(r)}(1)}\right)^2-z^{n-j+1}\left(\frac{\widehat P_{j-1}^{(r)}(z)}{\widehat P_n^{(r)}(z)}\right)^2\right].
\end{align*}
Note that, for any $z\in \mathcal{D}$ (hence $|z|\leq 1$) and any $2\leq j\leq n$, we have
\begin{align*}
\left|z^{n-j+1}\left(\frac{\widehat P_{j-1}^{(r)}(z)}{\widehat P_n^{(r)}(z)}\right)^2\right|\leq\left|\left(\frac{\widehat P_{j-1}^{(r)}(z)}{\widehat P_n^{(r)}(z)}\right)^2\right|<\left(\frac{\widehat P_{j-1}^{(r)}(1)}{\widehat P_n^{(r)}(1)}\right)^2,
\end{align*}
where the second inequality follows from Lemma \ref{lem:F_properties}. 
Hence, for any $z\in\mathcal{D}$ there holds
\begin{align*}
 \operatorname{Re}\left[
\left( \frac{\widehat{P}_{j-1}^{(r)}(1)}{\widehat{P}_n^{(r)}(1)} \right)^2-z^{n-j+1}
\left( \frac{\widehat{P}_{j-1}^{(r)}(z)}{\widehat{P}_n^{(r)}(z)} \right)^2
\right]>0.
\end{align*}
Combining  this inequality with \eqref{eq_rho} and \eqref{eq:rho_difference}, we deduce \eqref{Re>0}.
Thus the real part of $\mathcal{Q}_n({P^{(r)};z})$ is greater than $1$ over $\mathcal{D}$ for $n\geq3$. Together with the case $n=2$, this completes the proof.
\end{proof}

Now we are in a position to apply Proposition~\ref{prop_main result} to prove Theorem~\ref{prop_second order}.

\begin{proof}[Proof of Theorem~\ref{prop_second order}]
We apply Proposition~\ref{prop_main result} with
$\Omega=\mathcal{H}:=\{z\in\mathbb{C}:\operatorname{Re}z\geq0,\ z\neq0\}$.
As we mentioned at the beginning of this section, all $N_n^{(r)}(z)$ are $\mathcal{H}$-stable. 
This fact can also be obtained from Lemma \ref{lem:F_properties} due to its equivalence with the $\mathcal{D}$-stability of $P_{n}^{(r)}(z)$. 

It remains to verify the modulus condition for 
$\mathcal{Q}_n(N^{(r)};z)$ over $\mathcal{H}$. By Lemma \ref{lem:L_estimate}, we see that $|\mathcal{Q}_n(P^{(r)};u)|=|\mathcal{Q}_n(N^{(r)};z)|>1$ over $\mathcal{H}$ for any $n\geq 2$. Thus, we only need to check that $|\mathcal{Q}_1(N^{(r)};z)|>1$. Since $N_0^{(r)}(z)=1$, $N_1^{(r)}(z)=z$, and $N_2^{(r)}(z)=z(1+z)$, we get
\begin{align*}
|\mathcal{Q}_1(N^{(r)};z)|^2=\left|\frac{1+z}{z}\right|^2
=1+\frac{1+2\operatorname{Re}z}{|z|^2}>1
\end{align*}
for every $z\in\mathcal{H}$, as desired. The proof is complete. 
\end{proof}

As a consequence of Theorem \ref{prop_second order}, we have the following result at once, which implies Theorem~\ref{Narayana polynomial thm}. 

\begin{cor}
\label{cor_Narayana}
For any pair of partitions $\lambda$ and $\mu$ 
with $\lambda\unrhd\mu$, the Tur\'an expression
$\mathcal{T}_{\lambda,\mu}(N;z)$ is a generalized stable polynomial.
\end{cor}

Combining Proposition~\ref{prop_main result}, Lemma \ref{lem:F_properties}, Lemma \ref{lem:L_estimate} and  the fact that $\mathcal{Q}_1(P^{(r)};z)\equiv 1$ over $\mathcal{D}$, we immediately
obtain the following result, which is parallel to Theorem~\ref{prop_second order} and whose proof is omitted here. 

\begin{thm}\label{Pn polynomial thm}
Let $(P_n^{(r)}(z))_{n\geq0}$ be the polynomial sequence defined by \eqref{eq:F_recurrence} for some $r>1/2$.
If $\lambda,\mu$ are two partitions satisfying $\lambda\unrhd\mu$, then $\mathcal{T}_{\lambda,\mu}(P^{(r)};z)$ is $\mathcal{D}$-stable.
\end{thm}

\noindent \textbf{Acknowledgements.} 
This work is supported by the Fundamental Research Funds for the Central Universities.
We are grateful to Lily Liu, Xue Yan, Arthur Yang and Bao-Xuan Zhu
for their helpful discussions and  suggestions.

\end{document}